\documentclass{amsart}

\usepackage{amsmath, amscd} 
\usepackage{amsthm} 
\usepackage{amssymb}
\usepackage{amsfonts} 
\usepackage{qsymbols} 
\usepackage{latexsym}
\usepackage{cite}
\usepackage{url}

\newcommand {\R}{{\mathbb R}} 

\newcommand {\C}{{\mathbb C}} 
 
\newcommand {\N}{{{\mathbb N}}}

\newcommand {\B}{{\mathrm{B}}}
\newcommand{\Ce}{\mathrm{C}}
\newcommand {\D}{{\mathrm{D}}}

\newcommand {\E}{{\mathcal{E}}} 
\newcommand {\F}{{\mathcal{F}}}
\newcommand {\HT}{{\mathrm{H}}} 
\newcommand {\I}{{\mathrm{I}}}
\newcommand{\Ell}{\mathrm{L}} 
\newcommand {\La}{{\mathcal{L}}}
\newcommand{\eM}{\mathrm{M}}
\newcommand{\Ma}{{\mathcal{M}}} 
\newcommand {\Ri}{{\mathrm{R}}} 
\newcommand{\Sw}{\mathcal{S}}
\newcommand{\Se}{\mathrm{S}}

\newcommand {\W}{{\mathrm{W}}}
\newcommand{\al}{{\alpha}}
\newcommand{\ud}{\mathrm{d}} 
\newcommand{\ue}{\mathrm{e}} 
\newcommand{\ui}{\mathrm{i}} 
\newcommand {\ph}{{\varphi}} 
\newcommand {\w}{{\omega}}

\newcommand{\St}{{\mathrm{St}}} 
\newcommand{\supp}{{\mathrm{supp}}} 
\newcommand{\abs}[1]{\lvert#1\rvert}
\newcommand{\norm}[1]{\left\|#1\right\|}
\newcommand{\vanish}[1]{\relax}

\DeclareMathOperator{\Real}{Re}
\DeclareMathOperator{\Imag}{Im}

\newtheorem{theorem}{Theorem}[section]
\newtheorem{lemma}[theorem]{Lemma}
\newtheorem{proposition}[theorem]{Proposition}
\newtheorem{corollary}[theorem]{Corollary}

\theoremstyle{definition} 

\newtheorem{remark}[theorem]{Remark}

\numberwithin{equation}{section}

\title[Functional calculus on interpolation spaces]{Functional calculus on real interpolation spaces for generators of $C_{0}$-groups}

\author{Markus Haase}
\address{Delft Institute of Applied
Mathematics\\Mekelweg 4\\2628CD Delft\\The Netherlands}
\email{m.h.a.haase@tudelft.nl}

\author{Jan Rozendaal}
\address{Delft Institute of Applied Mathematics\\Mekelweg 4\\2628CD Delft\\The Netherlands}
\email[corresponding author]{janrozendaalmath@gmail.com}
\thanks{The second-named author is supported by NWO-grant
613.000.908 ``Applications of Transference Principles".}


\date{\today}

\dedicatory{}


\begin{document}

\begin{abstract}
We study functional calculus properties of $C_{0}$-groups on real interpolation spaces, using transference principles. We obtain interpolation versions of the classical transference principle for bounded groups and of a recent transference principle for unbounded groups. Then we show that each group generator on a Banach space has a bounded $\HT^{\infty}_{1}$-calculus on real interpolation spaces. Additional results are derived from this.
\end{abstract}

\maketitle

\section{Introduction}\label{introduction}

The classical transference principle by Berkson, Gillespie and Muhly from \cite{Berkson-Gillespie-Muhly} yields an estimate 
\begin{align}\label{classical transference principle}
\norm{\int_{\R}U(s)x\,\mu(\ud s)}_{X}\leq M^{2}\norm{L_{\mu}}_{\La(\Ell^{p}(X))}\norm{x}_{X}
\end{align}
for all $x\in X$, where $(U(s))_{s\in\R}\subseteq \La(X)$ is a bounded $C_{0}$-group of operators on a Banach space $X$ with uniform bound $M$, $\mu$ is a complex Borel measure on $\R$ and $L_{\mu}$ is convolution with $\mu$ on $\Ell^{p}(X)$, the space of $p$-integrable $X$-valued functions, for $p\in[1,\infty]$. Under certain geometrical assumptions on $X$, the norm of $L_{\mu}$ can be bounded in terms of a suitable norm of the Fourier transform $\F\mu$ of $\mu$. For instance, if $X$ is a Hilbert space then $\norm{L_{\mu}}_{\La(\Ell^{2}(X))}$ is equal to $\norm{\F\mu}_{\infty}$, by Plancherel's theorem. If $X$ is a UMD space and $p\in(1,\infty)$ then bounds for $\norm{L_{\mu}}_{\La(\Ell^{p}(X))}$ follow from the Mikhlin multiplier theorem. By combining this with \eqref{classical transference principle}, functional calculus bounds for the generator $A$ of $(U(s))_{s\in\R}$ can be obtained, i.e., estimates of the form $\norm{f(A)}\leq C\norm{f}_{F}$ for all $f$ in some function algebra $F$. Such bounds are important for evolution equations, conform for instance \cite{Arendt04, Kalton-Weis1}.

Useful as this procedure is, the assumptions on the space $X$ restrict the generality of the results. In particular, Hilbert and UMD spaces are reflexive. Therefore the approach described above generally does not yield interesting results for groups of operators on non-reflexive spaces, such as $\Ce(K)$-spaces or $\Ell^{1}$-spaces. In this paper we take a different approach and consider transference principles on interpolation spaces. It is known that the functional calculus properties of operators improve upon restriction to interpolation spaces, conform for instance the result of Dore \cite{Dore1999} that an invertible sectorial operator has a bounded sectorial $\HT^{\infty}$-calculus on real interpolation spaces. However, we are interested in functional calculus on strips, the more natural choice for group generators. We use that on Besov spaces Fourier multiplier results hold that do not depend on the geometry of the underlying space \cite{Girardi-Weis, Hytonen}. Since Besov spaces are obtained from real interpolation between $\Ell^{p}$ and Sobolev spaces, this fits into the setting of a transference principle on interpolation spaces. In Proposition \ref{transference bounded groups} we derive the following version of \eqref{classical transference principle} on the real interpolation space $(X,\D(A))_{\theta,q}$ from \eqref{real interpolation}. For the $X$-valued Besov space $\B^{\theta}_{p,q}(X)$ see Section \ref{fourier multipliers on besov spaces}.


\begin{proposition}\label{transference bounded groups introduction}
Let $X$ be a Banach space and let $\theta\in(0,1)$, $p\in[1,\infty)$ and $q\in[1,\infty]$. Then there exists a constant $C\geq 0$ such that the following holds. If $-\ui A$ generate a $C_{0}$-group $(U(s))_{s\in\R}$ on a Banach space $X$ with $M:=\sup_{s\in\R}\norm{U(s)}<\infty$, then
\begin{align*}
\norm{\int_{\R}U(s)x\,\mu(\ud s)}_{(X,\D(A))_{\theta,q}}\leq  CM^{2}\norm{L_{\mu}}_{\La(\B^{\theta}_{p,q}(X))}\norm{x}_{(X,\D(A))_{\theta,q}}
\end{align*}
for all complex Borel meaures $\mu$ on $\R$ and $x\in(X,\D(A))_{\theta,q}$.
\end{proposition}

Combining Proposition \ref{transference bounded groups introduction} with the aforementioned Fourier multiplier results on Besov spaces yields the following, a consequence of Corollary \ref{functional calculus real line}.

\begin{corollary}\label{functional calculus real line introduction}
Let $-\ui A$ generate a uniformly bounded $C_{0}$-group $(U(s))_{s\in\R}$ on a Banach space $X$, and let $\theta\in(0,1)$, $q\in[1,\infty]$. Then there exists a constant $C\geq 0$ such that 
\begin{align*}
\norm{\int_{\R}U(s)x\,\mu(\ud s)}_{(X,\D(A))_{\theta,q}}\leq C\sup_{s\in\R}\,\abs{\F\mu(s)}+(1+\abs{s})\,\abs{(\F\mu)'(s)}\norm{x}_{(X,\D(A))_{\theta,q}}
\end{align*}
for all $x\in(X,\D(A))_{\theta,q}$ and for each $\mu\in\eM(\R)$ such that $\F\mu\in\Ce^{1}\!(\R)$ with $\sup_{s\in\R}(1+\abs{s})\,\abs{(\F\mu)'(s)}<\infty$.
\end{corollary}

We also obtain an interpolation version of the transference principle for unbounded groups from \cite{Haase5}, as Proposition \ref{transference unbounded groups}. In terms of functional calculus for the part of $A$ in $(X,\D(A))_{\theta,q}$, these transference principles yield a result for functions in the \emph{analytic Mikhlin algebra}
\begin{align}\label{analytic mikhlin algebra}
\HT^{\infty}_{1}\!(\St_{\w}):=\left\{f\in \HT^{\infty}\!(\St_{\w})\left| \sup_{z\in\St_{\w}}(1+\abs{z})\,\abs{f'(z)}<\infty\right.\right\},
\end{align}
endowed with the norm 
\begin{align}\label{mikhlin norm}
\norm{f}_{\HT^{\infty}_{1}\!(\St_{\w})}:=\sup_{z\in\St_{\w}}\abs{f(z)}+(1+\abs{z})\,\abs{f'(z)}\qquad(f\in\HT^{\infty}_{1}\!(\St_{\w})).
\end{align}
Here $\St_{\w}:=\left\{z\in\C\mid \abs{\Imag(z)}<\w\right\}$ for $\w>0$. Note that definition \eqref{mikhlin norm} of the norm in the analytic Mikhlin algebra is different from that in \cite{Haase5}, where the quantity $\norm{f}=\sup_{z\in\St_{\w}}\abs{f(z)}+\abs{zf'(z)}$ is considered. However, the two norms are equivalent on domains containing zero, and \eqref{mikhlin norm} is more natural in the setting of transference principles on (inhomogeneous) Besov spaces, since Fourier multiplier results on such spaces require an inhomogeneous condition at zero. See also Remarks \ref{comparison with UMD case} and \ref{original result unbounded groups}.

Our main functional calculus result is as follows. For the group type $\theta(U)$ see \eqref{group type}, and for a proof of this result see Theorem \ref{main functional calculus result}.


\begin{theorem}\label{main functional calculus result introduction}
Let $-\ui A$ be the generator of a $C_{0}$-group $(U(s))_{s\in\R}$ on a Banach space $X$, and let $\theta\in(0,1)$, $q\in[1,\infty]$. Then the part of $A$ in $(X,\D(A))_{\theta,q}$ has a bounded $\HT^{\infty}_{1}\!(\St_{\w})$-calculus for all $\w>\theta(U)$. If $(U(s))_{s\in\R}$ is uniformly bounded then the constant bounding the $\HT^{\infty}_{1}\!(\St_{\w})$-calculus does not depend on $\w>0$.
\end{theorem}

In \cite{Haase5}, Theorem \ref{main functional calculus result introduction} is obtained for group generators on UMD spaces and functional calculus for the operator $A$ itself. Our result shows that on interpolation spaces no assumptions on the geometry of the underlying space are required. This means that even on spaces which are not UMD, such as $\Ce(K)$-spaces and $\Ell^{1}$-spaces, one can obtain functional calculus results if one is willing to restrict to interpolation spaces. Moreover, our results reaffirm the philosophy that the functional calculus properties of an operator improve when restricted to interpolation spaces, as was already evidenced for functions on sectors by the result of Dore \cite[Theorem 3.2]{Dore1999}. 

From Theorem \ref{main functional calculus result introduction} we deduce other functional calculus statements, for sectorial operators and generators of cosine functions.

Section \ref{functional calculus and fourier multipliers} provides the necessary background on functional calculus and the theory of Fourier multipliers on Besov spaces. In Section \ref{transference principles} we establish transference principles on interpolation spaces, and in Section \ref{functional calculus results} we prove Theorem \ref{main functional calculus result introduction}. Section \ref{additional results} contains additional results that can be derived from this.

\subsection{Notation and terminology}\label{notation and terminology}

The natural numbers are $\N:=\left\{1,2,\ldots\right\}$, and we write $\N_{0}:=\N\cup\left\{0\right\}$. The letters $X$ and $Y$ denote Banach spaces over the complex number field, and $\La(X)$ is the Banach algebra of all bounded operators on $X$. The \emph{domain} $\D(A)\subseteq X$ of a closed unbounded operator $A$ on $X$ is a Banach space when endowed with the norm
\begin{align*}
\norm{x}_{\D(A)}:=\norm{x}+\norm{Ax}\qquad(x\in \D(A)).
\end{align*}
The \emph{spectrum} of $A$ is $\sigma(A)$ and the \emph{resolvent set} $\rho(A):=\C\setminus \sigma(A)$. For $z\in\rho(A)$ the operator $R(z,A):=(z \I-A)^{-1}\in\La(X)$ is the \emph{resolvent} of $A$ at $z$.

For $p\in[1,\infty]$, $\Ell^{p}(\R;X)$ is the Bochner space of equivalence classes of $X$-valued Lebesgue $p$-integrable functions on $\R$. The H\"{o}lder conjugate of $p\in[1,\infty]$ is $p'$ and is defined by $\frac{1}{p}+\frac{1}{p'}=1$. The norm on $\Ell^{p}(\R;X)$ is usually denoted by $\norm{\cdot}_{p}$. In the case $X=\C$ we will simply write $\Ell^{p}(\R)=\Ell^{p}(\R;\C)$.

By $\eM(\R)$ we denote the space of complex-valued Borel measures on $\R$ with the total variation norm. For $\w\geq 0$ we let $\eM_{\w}(\R)$ consist of those $\mu\in\eM(\R)$ of the form $\mu(\ud s)=\ue^{-\w \abs{s}}\nu(\ud s)$ for some $\nu\in\eM(\R)$, with
\begin{align*}
\norm{\mu}_{\eM_{\w}\!(\R)}:=\norm{\ue^{\w\abs{\cdot}}\mu}_{\eM(\R)}.
\end{align*}
Note that $\eM_{\w}(\R)$ is a Banach algebra under convolution. A function $g$ such that $[s\mapsto g(s)\, \ue^{\w\abs{s}}]\in \Ell^{1}(\R)$ is usually identified with its associated measure $\mu\in \eM_{\w}(\R)$ given by $\mu(\ud s)=g(s)\ud s$.

For $\Omega\neq \emptyset$ open in $\C$ we let $\HT^{\infty}\!(\Omega)$ be the unital Banach algebra of bounded holomorphic functions on $\Omega$ with the supremum norm 
\begin{align*}
\norm{f}_{\HT^{\infty}\!(\Omega)}:= \sup_{z\in\Omega}\,\abs{f(z)}\qquad(f\in\HT^{\infty}\!(\Omega)).
\end{align*}
We mainly consider the case where $\Omega$ is a strip of the form
\begin{align*}
\St_{\w}:=\left\{z\in \C\mid \abs{\Imag(z)}<\w\right\}
\end{align*}
for $\w>0$, with $\St_{0}:=\R$.

The \emph{Schwartz class} $\Sw(\R;X)$ is the space of $X$-valued rapidly decreasing smooth functions on $\R$, and the space of $X$-valued \emph{tempered distributions} is $\Sw'(\R;X)$. The \emph{Fourier transform} of an $X$-valued tempered distribution $\Phi\in \Sw'(\R;X)$ is denoted by $\F\Phi$. For instance, if $\mu\in\eM_{\w}(\R)$ for $\w>0$ then $\F\mu\in\HT^{\infty}\!(\St_{\w})\cap \Ce(\overline{\St_{\w}})$ is given by
\begin{align*}
\F\mu(z):=\int_{\R}\ue^{-\ui sz}\mu(\ud s)\qquad(z\in\St_{\w}).
\end{align*}

If $X$ and $Y$ are Banach spaces that are embedded continuously into a Hausdorff topological vector space $Z$, then we call $(X,Y)$ an \emph{interpolation couple}. We let 
\begin{align*}
K(t,z):=\inf\left\{\norm{x}_{X}+t\norm{y}_{Y}\mid x\in X, y\in Y, x+y=z\right\}
\end{align*}
for $t>0$ and $z\in X+Y\subseteq Z$. The \emph{real interpolation space} of $X$ and $Y$ with parameters $\theta\in [0,1]$ and $q\in [1,\infty]$ is
\begin{align}\label{real interpolation}
(X,Y)_{\theta,q}:=\left\{z\in X+Y\mid[t\mapsto t^{-\theta}K(t,z)]\in \Ell^{q}((0,\infty),\ud t/t)\right\},
\end{align}
a Banach space when equipped with the norm 
\begin{align*}
\norm{z}_{(X,Y)_{\theta,q}}:=
\norm{t\mapsto t^{-\theta}K(t,z)}_{\Ell^{q}((0,\infty),\ud t/t)}\qquad(z\in(X,Y)_{\theta,q}).
\end{align*}
If $T:X+Y\rightarrow X+Y$ restricts to a bounded operator on $X$ and a bounded operator on $Y$ then 
\begin{align}\label{interpolation inequality}
\norm{T}_{\La((X,Y)_{\theta,q})}\leq \norm{T}_{\La(X)}^{1-\theta}
\norm{T}_{\La(Y)}^{\theta}
\end{align}
for all $\theta\in (0,1)$ and $q\in[1,\infty]$ \cite[Theorem 3.1.2]{Bergh-Lofstrom}. We mainly consider interpolation spaces for the couple 
$\left(X,\D(A)\right)$, where $A$ is a closed operator on $X$. We write 
\begin{align*}
\D_{A}(\theta,q):=(X,\D(A))_{\theta,q}
\end{align*}
and
\begin{align*}
\norm{x}_{\theta,q}:=\norm{x}_{\D_{A}(\theta,q)}\qquad\quad (x\in\D_{A}(\theta,q)).
\end{align*}
For an operator $B$ on $X$ and a continuously embedded space $Y\hookrightarrow X$, the operator $B_{Y}$ on $Y$ that satisfies $B_{Y}y=By$ for elements in its domain
\begin{align*}
\D(B_{Y}):=\left\{y\in \D(B)\cap Y\mid By\in Y\right\}
\end{align*}
is the \emph{part} of $B$ in $Y$. If $Y=\D_{A}(\theta,q)$ for $\theta\in(0,1)$ and $q\in[1,\infty]$ then we write 
\begin{align*}
B_{\theta,q}:=B_{\D_{A}(\theta,q)}.
\end{align*}

Throughout, an $X$-valued function space $\Phi(\R;X)$ on the real line will be denoted by $\Phi(X)$ whenever little confusion can arise.

\section{Functional calculus and Fourier multipliers}\label{functional calculus and fourier multipliers}

\subsection{Functional calculus}\label{functional calculus}

We assume that the reader is familiar with the basics of the theory of $C_{0}$-groups as developed in, for instance, \cite{Engel-Nagel}, and merely recall some of the notions and results in functional calculus theory that are used. Details on functional calculus for group generators can be found in \cite[Chapter 4]{Haase1}.

Let $-\ui A$ be the generator of a $C_{0}$-group $(U(s))_{s\in\R}$ on a Banach space $X$. Then the \emph{group type} of $U$,
\begin{align}\label{group type}
\theta(U):=\inf\left\{\w\geq 0\left| \exists \text{$M\geq 1$
such that $\norm{U(s)}\leq M\ue^{\w\abs{s}}$ for all $s\geq 0$}\right.\right\},
\end{align}
is finite. Moreover, $A$ is a \emph{strip-type operator} of \emph{height} $\w_{0}:=\theta(U)$, i.e., $\sigma(A)\subseteq \overline{\St_{\w_{0}}}$ and
\begin{align*}
\sup_{\lambda\in\C\setminus \St_{\w}}\norm{R(\lambda,A)}<\infty\qquad\text{for all $\w>\w_{0}$}.
\end{align*}
The \emph{strip-type functional calculus} for $A$ is defined as follows. First, operators $f(A)\in \La(X)$ are associated with functions 
\begin{align*}
f\in \E(\St_{\w}):=\left\{g\in \HT^{\infty}\!(\St_{\w})\left| \text{$g(z)\in
O(\abs{z}^{-\alpha})$ for some $\alpha>1$ as $\abs{\Real(z)}\rightarrow \infty$}\right.\right\}
\end{align*}
for $\w>\w_{0}$, by a Cauchy-type integral
\begin{align*}
f(A):=\frac{1}{2\pi \ui}\int_{\delta \St_{\w'}}f(z)R(z,A)\,\ud z.
\end{align*}
Here $\delta \St_{\w'}$ is the positively oriented boundary of $\St_{\w'}$ for $\w'\in (\w_{0},\w)$. This procedure is independent of the choice of $\w'$ by Cauchy's theorem, and yields an algebra homomorphism $\E(\St_{\w})\rightarrow \La(X)$. The definition of $f(A)$ is extended to a larger class of functions by \emph{regularization}, i.e.
\begin{align*}
f(A):=e(A)^{-1}(ef)(A)
\end{align*}
if there exists $e\in \E(\St_{\w})$ with $e(A)$ injective and $ef\in \E(\St_{\w})$. This yields a closed unbounded operator $f(A)$ on $X$, and the definition of $f(A)$ is independent of the choice of the regularizer $e$. The algebra of all meromorphic functions on $\St_{\w}$ that are regularizable for $A$ is denoted by $\Ma_{A}(\St_{\w})$. Each $f\in \HT^{\infty}\!(\St_{\w})$ is regularizable by the function $z\mapsto(\lambda-z)^{-2}$, for $\abs{\Imag(\lambda)}>\w$.

Since $-\ui A$ generates a $C_{0}$-group, the \emph{Hille-Phillips functional calculus} for $A$ yields certain functions $f$ that give rise to bounded operators $f(A)$. Fix $M\geq 1$ and $\w\geq 0$ such that $\norm{U(s)}\leq M\ue^{\w\abs{s}}$ for all $s\in\R$. For $\mu\in\eM_{\w}(\R)$ define
\begin{align}\label{phillips calculus definition}
U_{\mu}x:=\int_{\R}U(s)x\,\mu(\ud s)\qquad(x\in X).
\end{align}
Then $\mu\mapsto U_{\mu}$ is an algebra homomorphism $\eM_{\w}(\R)\rightarrow \La(X)$. The following lemma, Lemma 2.2 in \cite{Haase4}, shows that the Hille-Phillips calculus extends the strip-type calculus for $A$.

\begin{lemma}\label{decay implies fourier transform}
Let $X$, $A$ and $U$ be as above, and let $\w'>\w\geq 0$.
\begin{itemize}
\item[a)] For each $f\in \E(\St_{\w'})$ there exists $\mu\in \eM_{\w}(\R)$ such that $f=\F\mu$.
\item[b)] Let $\mu\in\eM_{\w}(\R)$ be such that $\F\mu$ extends to an element of $\Ma_{A}(\St_{\w'})$. Then $f(A)=U_{\mu}\in\La(X)$ and 
					\begin{align*}
					\sup_{t\in\R}\norm{f(t+A)}\leq M\norm{\mu}_{\eM_{\w}(\R)}.
					\end{align*}
\end{itemize}
\end{lemma}

We now consider functional calculus for operators on interpolation spaces. The following lemma shows that, in particular, the functional calculi for $A$ and $A_{\theta,q}$ are compatible.

\begin{lemma}\label{functional calculus on interpolation spaces}
Let $A$ be a strip-type operator of height $\w_{0}$ on a Banach space $X$ and let $\theta\in(0,1)$, $q\in[1,\infty]$ and $m,n\in\N_{0}$. Let $Y:=(\D(A^{m}),\D(A^{n}))_{\theta,q}$.
\begin{itemize}
\item[a)] The part $A_{Y}$ of $A$ in $Y$ is a strip-type operator of height $\w_{0}$. Moreover, $f\in \Ma_{A_{Y}}(\St_{\w})$ with $f(A_{Y})=f(A)_{Y}$ for all $\w>\w_{0}$ and $f\in\Ma_{A}(\St_{\w})$.
\item[b)] If $-\ui A$ generates a $C_{0}$-group $(U(s))_{s\in\R}$ on $X$ and $q<\infty$, then $-\ui A_{Y}$ generates the $C_{0}$-group $(U(s)_{Y})_{s\in\R}$. In particular, $\D(A_{Y})$ is dense in $Y$.
\end{itemize}
\end{lemma}
\begin{proof}

a) First note that, for all $k\in\N_{0}$ and $\lambda\in \rho(A)$, $R(\lambda,A)$ leaves $\D(A^{k})$ invariant with $\norm{R(\lambda,A)}_{\mathcal{L}(\D(A^{k}))}\leq \norm{R(\lambda,A)}_{\mathcal{L}(X)}$. By \eqref{interpolation inequality}, $R(\lambda,A)$ leaves $Y$ invariant with 
\begin{align}\label{resolvent estimate}
\norm{R(\lambda,A)}_{\mathcal{L}(Y)}\leq \norm{R(\lambda,A)}_{\mathcal{L}(X)}.
\end{align}
By \cite[Proposition A.2.8]{Haase1}, $\sigma(A_{Y})\subseteq \sigma(A)$ and $R(\lambda,A_{Y})=R(\lambda,A)_{Y}$ for all $\lambda\in \rho(A)$. Hence \eqref{resolvent estimate} yields the first statement. Let $\w>\w_{0}$ and $f\in \E(\St_{\w})$ be given. Then
\begin{align*}
f(A_{Y})y=\frac{1}{2\pi \ui}\int_{\Gamma}f(z)R(z,A_{Y})y\,\ud z=\frac{1}{2\pi \ui}\int_{\Gamma}f(z)R(z,A)y\,\ud z=f(A)y
\end{align*}
for some contour $\Gamma$ and all $y\in Y$. For a general $f\in\mathcal{M}_{A}(\St_{\w})$, note that $e$ is a regulariser for $f$ in the functional calculus for $A_{Y}$ if it is a regulariser for $f$ in the functional calculus for $A$, since then $e(A_{Y})=e(A)_{Y}$ is injective. The rest follows by regularization.

b) By \eqref{interpolation inequality}, $\norm{U(s)_{Y}}\leq \norm{U(s)}$ for all $s\in\R$. Hence $(U(s)_{Y})_{s\in\R}$ is locally bounded. Since it is strongly continuous on the dense subset $\D(A^{\max(n,m)})\subseteq Y$ \cite[Proposition 1.2.5]{Lunardi}, it is strongly continuous on $Y$. By \cite[p.~60]{Engel-Nagel}, $-\ui A_{Y}$ is its generator. 
\end{proof}

\begin{remark}\label{types of convergence}
Part b) of Lemma \ref{functional calculus on interpolation spaces} ensures that the integral in \eqref{phillips calculus definition} is well-defined and converges in $\D_{A}(\theta,q)$, for $x\in\D_{A}(\theta,q)$ and $q<\infty$. Even though $(U(s))_{s\in\R}$ is not strongly continuous on $\D_{A}(\theta,\infty)$ in general, the integral is well-defined and converges in $X$. Since $\D_{A}(\theta,q)$ is continuously embedded in $X$ for all $\theta\in(0,1)$ and $q\in[1,\infty]$, the value of the integral does not depend on the space in which convergence takes place. Hence from now on we regularly will not specify in which norm \eqref{phillips calculus definition} converges.
\end{remark}

Let $A$ be a strip-type operator of height $\w_{0}$ and $\w>\w_{0}$. For a Banach algebra $F$ of functions that is continuously embedded in $\HT^{\infty}\!(\St_{\w})$, we say that $A$ has a \emph{bounded $F$-calculus} if there exists a constant $C\geq 0$ such that $f(A)\in \La(X)$ with
\begin{align*}
\norm{f(A)}_{\La(X)}\leq C\norm{f}_{F}\qquad\text{for all $f\in F$}.
\end{align*}

The next lemma from \cite[Proposition 5.1.7]{Haase1} is fundamental. 

\begin{lemma}[Convergence Lemma]\label{convergence lemma}
Let $A$ be a densely defined strip-type operator of height $\w_{0}$ on a Banach space $X$. Let $\w>\w_{0}$ and $(f_{j})_{j\in J}\subseteq \HT^{\infty}\!(\St_{\w})$ be a net satisfying the following conditions:
\begin{enumerate}
\item $\sup_{j\in J}\norm{f_{j}}_{\HT^{\infty}\!(\St_{\w})}<\infty$;
\item $f(z):=\lim_{j}f_{j}(z)$ exists for all $z\in \St_{\w}$;
\item $\sup_{j\in J}\norm{f_{j}(A)}_{\La(X)}<\infty$.
\end{enumerate}
Then $f\in\HT^{\infty}\!(\St_{\w})$, $f(A)\in\La(X)$, $f_{j}(A)\rightarrow f(A)$ strongly and
\begin{align*}
\norm{f(A)}\leq \limsup_{j\in J}\norm{f_{j}(A)}.
\end{align*}
\end{lemma}

\subsection{Fourier multipliers on Besov spaces}\label{fourier multipliers on besov spaces}

Let us summarize some results about Fourier multipliers on vector-valued Besov spaces which will be used later on. Details can be found in \cite{Amann} and \cite{Girardi-Weis}.

Let $\psi\in \Ce^{\infty}\!(\R)$ be a nonnegative function with support in $[\frac{1}{2},2]$ such that
\begin{align*}
\sum_{k=-\infty}^{\infty}\psi(2^{-k}s)=1\qquad\text{for all $s\in(0,\infty)$}.
\end{align*}
For $k\in\N$ and $s\in\R$ let $\ph_{k}(s):=\psi(2^{-k}\abs{s})$, and let $\ph_{0}(s):=1-\sum_{k=1}^{\infty}\ph_{k}(s)$. Let $X$ be a Banach space and let $p,q\in[1,\infty]$ and $r\in\R$ be given. The \emph{(inhomogeneous) Besov space} $\B^{r}_{p,q}(\R;X)$ consists of all $X$-valued tempered distributions $f\in\Sw'(\R;X)$ such that
\begin{align*}
\norm{f}_{\B^{r}_{p,q}(\R;X)}:=\norm{\Big(2^{kr}\big\|\F^{-1}\ph_{k}\ast f\big\|_{\Ell^{p}(\R;X)}\Big)_{k=0}^{\infty}}_{\ell^{q}}<\infty,
\end{align*}
endowed with the norm $\norm{f}_{\B^{r}_{p,q}(\R;X)}$. Then $\B^{r}_{p,q}(\R;X)$ is a Banach space such that $\Sw(\R;X)\subseteq \B^{r}_{p,q}(\R;X)$, and a different choice of $\psi$ leads to an equivalent norm on $\B^{r}_{p,q}(\R;X)$.

For $n\in\N$ and $p\in[1,\infty]$ the Sobolev space
\begin{align*}
\W^{n,p}(\R;X):=\left\{f\in \Ell^{p}(\R;X)\mid \text{$f^{(k)}\in \Ell^{p}(\R;X)$ for all $1\leq k\leq n$}\right\},
\end{align*}
is a Banach space when endowed with the norm
\begin{align*}
\norm{f}_{n,p}:=\norm{f}_{\W^{n,p}(X)}:=\|f\|_{p}+\|f^{(n)}\|_{p}\qquad(f\in \W^{n,p}(\R;X)).
\end{align*}
In the case $X=\C$ we simply write $\W^{n,p}(\R)=\W^{1,p}(\R;\C)$.

The following lemma is equation (5.9) in \cite{Amann}. The fact that the constant $C$ does not depend on the particular Banach space follows from a direct sum argument.

\begin{lemma}\label{besov spaces interpolation}
Let $\theta\in(0,1)$, $p\in[1,\infty)$, $q\in[1,\infty]$ and $n\in\N$. Then there exists a constant $C>0$ such that, for any Banach space $X$, $(\Ell^{p}(X),\W^{n,p}(X))_{\theta,q}=\B^{n\theta}_{p,q}(X)$ with
\begin{align*}
\frac{1}{C}\norm{f}_{\B^{n\theta}_{p,q}(X)}\leq \norm{f}_{(\Ell^{p}(X),\W^{n,p}(X))_{\theta,q}}\leq C\norm{f}_{\B^{n\theta}_{p,q}(X)}\qquad(f\in\B^{n\theta}_{p,q}(X)).
\end{align*}
\end{lemma}

Let $m\in \Ell^{\infty}\!(\R;\La(X))$, $p,q\in[1,\infty]$ and $r\in \R$. We say that $m$ is a \emph{bounded Fourier multiplier} on $\B^{r}_{p,q}(X)$ if there is a unique bounded operator $T_{m}:\B^{r}_{p,q}(X)\rightarrow \B^{r}_{p,q}(X)$ such that
\begin{align}\label{besov multiplier}
T_{m}(f)=\F^{-1}\left(m\cdot\F f\right)
\end{align}
for all $f\in\Sw(X)$. Each $\mu\in\eM(\R)$ induces a bounded Fourier multiplier $\F\mu$ with 
\begin{align}\label{convolution}
T_{\F\mu}(f)=L_{\mu}(f):=\mu\ast f\qquad(f\in\Sw(X)).
\end{align}
The main result about Fourier multipliers on Besov spaces that we use is the following, Corollary 4.15 from \cite{Girardi-Weis}.

\begin{proposition}\label{besov multiplier theorem}
There exists a constant $C\geq0$ such that the following holds. Let $X$ be a Banach space, $p,q\in[1,\infty]$ and $r\in\R$. If $m:\R\rightarrow \C$ is such that $\ph_{k}m\in\B^{1/2}_{2,1}(\R;\C)$ for all $k\in\N_{0}$, and
\begin{align*}
M:=\sup_{k\in\N_{0}}\inf_{a>0}\norm{(\ph_{k}m)(a\cdot)}_{\B^{1/2}_{2,1}(\R;\C)}<\infty,
\end{align*}
then $m$ is a bounded Fourier multiplier on $\B^{r}_{p,q}(X)$ with $\norm{T_{m}}_{\La(\B^{r}_{p,q}(X))}\leq CM$.
\end{proposition}

\begin{corollary}\label{mikhlin condition}
There exists a constant $C\geq0$ such that for all Banach spaces $X$, $p,q\in[1,\infty]$, $r\in\R$ and all $m\in \Ce^{1}(\R;\C)$ with
\begin{align*}
N:=\sup_{s\in\R}\,\abs{m(s)}+(1+\abs{s})\abs{m'(s)}<\infty,
\end{align*}
$m$ is a bounded Fourier multiplier on $\B^{r}_{p,q}(X)$ with $\norm{T_{m}}_{\La(\B^{r}_{p,q}(X))}\leq CN$.
\end{corollary}
\begin{proof} 
This follows as in \cite[Corollary 4.11]{Girardi-Weis}. See also \cite[Remark 4.16]{Girardi-Weis}.
\end{proof}

\section{Transference principles}\label{transference principles}

\subsection{Unbounded groups}\label{unbounded groups}

We first establish an interpolation version of the transference principle for unbounded groups from \cite{Haase5}. Note that, for each $\mu\in\eM(\R)$ and $p\in[1,\infty]$, the convolution operator $L_{\mu}$ from \eqref{convolution} extends to a bounded operator on $\Ell^{p}(X)$, by Young's inequality. For $\w\geq 0$ and $\mu\in\eM_{\w}(\R)$ let $\mu_{\w}\in \eM(\R)$ be given by $\mu_{\w}(\ud s):=\cosh(\w s)\mu(\ud s)$.

\begin{proposition}\label{transference unbounded groups}
Let $0\leq\w_{0}<\w$, $\theta\in(0,1)$, $p\in[1,\infty)$ and $q\in[1,\infty]$. Then there exists a constant $C\geq0$ such that the following holds. If $-\ui A$ generates a $C_{0}$-group $(U(s))_{s\in\R}$ on a Banach space $X$ such that $\norm{U(s)}_{\La(X)}\leq M\cosh(\w_{0}s)$ for all $s\in\R$ and some $M\geq 1$, then
\begin{align*}
\norm{\int_{\R}U(s)x\,\mu(\ud s)}_{\theta,q}\leq  CM^{2}\norm{L_{\mu_{\w}}}_{\La(\B^{\theta}_{p,q}(X))}\norm{x}_{\theta,q}
\end{align*}
for all $\mu\in\eM_{\w}(\R)$ and $x\in\D_{A}(\theta,q)$.
\end{proposition}
\begin{proof}
Let $\mu\in\eM_{\w}(\R)$ be given and let $U_{\mu}$ be as in \eqref{phillips calculus definition}. By the proof of Theorem 3.2 in\cite{Haase5}, we can factorize $U_{\mu}$ as $U_{\mu}=P\circ L_{\mu_{\w}}\circ \iota$, where
\begin{itemize}
	\item $\iota:X\rightarrow \Ell^{p}(X)$ is given by 
	\begin{align*}
	\iota x(s):=\psi(-s)U(-s)x\qquad (x\in X,s\in\R),
	\end{align*}
	with
	\begin{align*}
	\psi(s):=\frac{1}{\cosh(\alpha s)}\qquad(s\in\R)
	\end{align*}
	for $\alpha>\w$ fixed.
	\item $P:\Ell^{p}(X)\rightarrow X$ is given by 
	\begin{align*}
	Pf:=\int_{\R}\ph(s)U(s)f(s)\,\ud s\qquad(f\in \Ell^{p}(X)),
	\end{align*}
	with 
	\begin{align*}
	\ph(s):=\frac{\sqrt{8}\w}{\pi}\frac{\cosh(\w s)}{\cosh(2\w s)}\qquad(s\in\R).
	\end{align*}
\end{itemize}
Then, using H\"{o}lder's inequality,
\begin{align}
\label{norm iota on X}
\norm{\iota}_{\La(X,\Ell^{p}(X))}&\leq M\norm{\psi\cosh(\w_{0}\cdot)}_{p},\\
\label{norm P on X}
\norm{P}_{\La(\Ell^{p}(X),X)}&\leq M \norm{\ph\cosh(\w_{0}\cdot)}_{p'}.
\end{align}
We claim that $\iota:\D(A)\rightarrow \W^{1,p}(X)$ and $P:\W^{1,p}(X)\rightarrow \D(A)$ are well-defined and bounded. To prove this claim, first let $x\in\D(A)$. Then $\iota x\in\Ce^{1}\!(\R)$ with
\begin{align*}
(\iota x)'(s)&=-\psi'(-s)U(-s)x+\ui\psi(-s)U(-s)Ax\\
&=-\alpha\frac{\tanh(\alpha s)}{\cosh(\alpha s)}U(-s)x+\ui \frac{1}{\cosh(\alpha s)}U(-s)Ax
\end{align*}
for all $s\in\R$. Hence $(\iota x)'\in\Ell^{p}(\R)$ with
\begin{align*}
\norm{(\iota x)'}_{p}\leq \alpha M\norm{\tanh}_{\Ell^{\infty}(\R)}\norm{\frac{\cosh(\w_{0}\cdot)}{\cosh(\alpha\cdot)}}_{p}\!\norm{x}_{X}+M\norm{\frac{\cosh(\w_{0}\cdot)}{\cosh(\alpha\cdot)}}_{p}\!\norm{Ax}_{X}.
\end{align*}
Combining this with \eqref{norm iota on X} implies that $\iota x\in\W^{1,p}(\R)$ with
\begin{align}\label{norm iota on D(A)}
\norm{\iota x}_{1,p}\leq M(\alpha\norm{\tanh}_{\Ell^{\infty}(\R)}+1)\norm{\frac{\cosh(\w_{0}\cdot)}{\cosh(\alpha\cdot)}}_{p}\!\norm{x}_{\D(A)}.
\end{align}
This shows that $\iota:\D(A)\rightarrow \W^{1,p}(X)$ is bounded. To prove the claim for $P$, fix $f\in \Sw(X)$ and note that
\begin{align*}
\frac{1}{h}(U(h)-\I)Pf=\int_{\R}U(s)\frac{\ph(s-h)f(s-h)-\ph(s)f(s)}{h}\,\ud s
\end{align*}
for $h>0$. The latter expression converges to $-\int_{\R}U(s)(\ph f)'(s)\,\ud s\in X$ as $h\rightarrow 0$, by the dominated convergence theorem. Hence $Pf\in\D(A)$ with
\begin{align*}
APf=\lim_{h\rightarrow 0}\frac{1}{h}(U(h)-\I)Pf=-\int_{\R}U(s)(\ph'(s) f(s)+\ph(s)f'(s))\,\ud s.
\end{align*}
Another application of H\"{o}lder's inequality yields
\begin{align*}
\norm{APf}_{X}\leq M\norm{\ph'\cosh(\w_{0}\cdot)}_{p'}\norm{f}_{p}+M\norm{\ph\cosh(\w_{0}\cdot)}_{p'}\norm{f'}_{p}.
\end{align*}
Combining this with \eqref{norm P on X} implies
\begin{align}\label{norm P on D(A)}
\norm{Pf}_{\D(A)}\leq M\left(\norm{\ph \cosh(\w_{0}\cdot)}_{p'}+\norm{\ph'\cosh(\w_{0}\cdot)}_{p'}\right)\norm{f}_{1,p}.
\end{align}
As $\Sw(X)$ is dense in $\W^{1,p}(X)$, $P:\W^{1,p}(X)\rightarrow \D(A)$ is bounded.

Since $L_{\mu_{\w}}\in\La(\W^{1,p}(X))$, we can factorize $U_{\mu}\in\La(\D(A))$ as $U_{\mu}=P\circ L_{\mu_{\w}}\circ \iota$ via bounded maps through $\W^{1,p}(X)$. Applying the real interpolation method with parameters $\theta$ and $q$ to the two factorizations of $U_{\mu}$, through $\Ell^{p}(X)$ respectively $\W^{1,p}(X)$, yields the commutative diagram of bounded maps
\begin{align*}
\begin{CD}
			(\Ell^{p}(X),\W^{1,p}(X))_{\theta,q}		@>L_{\mu_{\w}}>>			(\Ell^{p}(X),\W^{1,p}(X))_{\theta,q}\\
			@AA\iota A																										@VVPV\\
			\D_{A}(\theta,q)									 			@>U_{\mu}>>						\D_{A}(\theta,q)
\end{CD}
\end{align*}
Finally, estimate the norms of $\iota$ and $P$ in this diagram by applying \eqref{interpolation inequality} to \eqref{norm iota on X} and \eqref{norm iota on D(A)} respectively \eqref{norm P on X} and \eqref{norm P on D(A)}. This yields
\begin{align}\label{besovless estimate unbounded}
\norm{U_{\mu}}_{\La(\D_{A}(\theta,q))}\leq C'M^{2}\norm{L_{\mu}}_{\La((\Ell^{p}(X),\W^{1,p}(X))_{\theta,q})}
\end{align}
for a constant $C'\geq 0$ independent of $\mu$. Now Lemma \ref{besov spaces interpolation} concludes the proof.
\end{proof}

\subsection{Bounded groups}\label{bounded groups}

In this section we establish a version of the classical transference principle from \cite{Berkson-Gillespie-Muhly} on interpolation spaces, already stated in the Introduction as Proposition \ref{transference bounded groups introduction}. In the proof we use the convention $1/\infty:=0$.

\begin{proposition}\label{transference bounded groups}
Let $\theta\in(0,1)$, $p\in[1,\infty)$ and $q\in[1,\infty]$. Then there exists a constant $C\geq 0$ such that the following holds. If $-\ui A$ generate a $C_{0}$-group $(U(s))_{s\in\R}$ on a Banach space $X$ with $M:=\sup_{s\in\R}\norm{U(s)}<\infty$, then
\begin{align}\label{transference estimate bounded groups}
\norm{\int_{\R}U(s)x\,\mu(\ud s)}_{\theta,q}\leq CM^{2}\norm{L_{\mu}}_{\La(\B^{\theta}_{p,q}(X))}\norm{x}_{\theta,q}
\end{align}
for all $\mu\in\eM(\R)$ and $x\in\D_{A}(\theta,q)$.
\end{proposition}
\begin{proof}
First note that it suffices to establish \eqref{transference estimate bounded groups} for measures with compact support. Indeed, approximating by measures with compact support then extends \eqref{transference estimate bounded groups} to all $\mu\in\eM(\R)$. So fix $N>0$ and let $\mu\in\eM(\R)$ be such that $\supp(\mu)\subseteq[-N,N]$. We will factorize $U_{\mu}$ using the abstract transference principle from \cite[Section 2]{Haase6}. To this end, let $\rho\in\Ce^{\infty}\!(\R)$ be defined by
\begin{align*}
\rho(s):=\left\{
\begin{array}{ll}
c_{1}\exp\left(\frac{1}{s^{2}-1}\right)&|s|<1\\
0&|s|\geq 1
\end{array}\right.,
\end{align*}
where $c_{1}\geq 0$ is such that $\int_{\R}\rho(s)\,\ud s=1$. Fix $\alpha,\beta>0$ and define $\sigma(s):=\frac{1}{\al}\,\rho\left(\frac{s}{\al}\right)$ for $s\in\R$, and
\begin{align*}
\psi:=\sigma\ast\mathbf{1}_{[-(N+3\alpha+\beta),N+3\alpha+\beta]}\qquad\text{and}\qquad \ph:=\frac{1}{2(\al+\beta)}\sigma\ast\mathbf{1}_{[-(\alpha+\beta),\alpha+\beta]}.
\end{align*}
Then $\psi,\ph\in \Ce^{\infty}\!(\R)$ are such that $\supp(\ph)\subseteq [-(2\alpha+\beta),2\alpha+\beta]$, 
\begin{align*}
\psi\equiv 1 \text{ on }[-(2\alpha+N+\beta),2\alpha+N+\beta]\quad\text{and}\quad \int_{-(2\alpha+\beta)}^{2\alpha+\beta}\ph(s)\,\ud s=1.
\end{align*}
Hence $\psi\ast \ph\equiv 1$ on $[-N,N]$. Let $\iota:X\rightarrow \Ell^{p}(X)$ be given by
\begin{align*}
\iota x(s):=\psi(-s)U(-s)x\qquad(x\in X, s\in\R),
\end{align*}
and $P:\Ell^{p}(X)\rightarrow X$ by
\begin{align*}
Pf:=\int_{\R}\ph(s)U(s)f(s)\,\ud s\qquad(f\in \Ell^{p}(X)).
\end{align*}
Proposition 2.3 in \cite{Haase6} yields the factorization $U_{\mu}=P\circ L_{\mu}\circ \iota$, where we use that $(\psi\ast\ph)\mu=\mu$.
By H\"{o}lder's inequality,
\begin{align}\label{norms on X bounded}
\norm{\iota}_{\La(X,\Ell^{p}(X))}\leq M\norm{\psi}_{p}\quad\textrm{and}\quad\norm{P}_{\La(\Ell^{p}(X),X)}\leq M\norm{\ph}_{p'}
\end{align}
Moreover, $\iota:\D(A)\rightarrow \W^{1,p}(X)$ and $P:\W^{1,p}(X)\rightarrow \D(A)$ are bounded with 
\begin{align}\label{norms on D(A) bounded}
\norm{\iota}_{\La(\D(A),\W^{1,p}(X))}\leq M\norm{\psi}_{1,p}\quad\textrm{and}\quad\norm{P}_{\La(\W^{1,p}(X),\D(A))}\leq M\norm{\ph}_{1,p'}.
\end{align}
This follows by arguments almost identical to those in the proof of Proposition \ref{transference unbounded groups}. Applying the real interpolation method with parameters $\theta$ and $q$ to the two factorizations of $U_{\mu}$, through $\Ell^{p}(X)$ and $\W^{1,p}(X)$, produces the commutative diagram of bounded maps
\begin{align*}
\begin{CD}
			(\Ell^{p}(X),\W^{1,p}(X))_{\theta,q}		@>L_{\mu}>>						(\Ell^{p}(X),\W^{1,p}(X))_{\theta,q}\\
			@AA\iota A																										@VVPV\\
			\D_{A}(\theta,q)									 			@>U_{\mu}>>						\D_{A}(\theta,q)
\end{CD}
\end{align*}
Use \eqref{interpolation inequality} on \eqref{norms on X bounded} and \eqref{norms on D(A) bounded} to estimate the norms of $\iota$ and $P$ in this factorization as $\norm{\iota}\leq M\norm{\psi}_{1,p}$ and $\norm{P}\leq M\norm{\ph}_{1,p'}$. This yields
\begin{align}\label{norm estimate on interpolation space}
\norm{U_{\mu}}_{\La(\D_{A}(\theta,q))}\leq M^{2}\norm{\psi}_{1,p}\norm{\ph}_{1,p'}\norm{L_{\mu}}_{\La((\Ell^{p}(X),\W^{1,p}(X))_{\theta,q})}.
\end{align}
To determine $\norm{\psi}_{1,p}$ and $\norm{\ph}_{1,p'}$, note that
\begin{align*}
\norm{\psi}_{p}&\leq \norm{\sigma}_{1}\norm{\mathbf{1}_{[-(N+3\al+\beta),N+3\al+\beta]}}_{p}=(2(N+3\al+\beta))^{1/p},\\
\norm{\ph}_{p'}&\leq \frac{1}{2(\alpha+\beta)}\norm{\sigma}_{1}\norm{\mathbf{1}_{[-(\al+\beta),\al+\beta]}}_{p'}=(2(\al+\beta))^{-1/p},
\end{align*}
by Young's inequality. Since $\sigma$ is an even function that is decreasing on $[0,\alpha]$ and supported on $[-\alpha,\alpha]$, its derivative satisfies 
\begin{align*}
\norm{\sigma'}_{1}=-2\int_{0}^{\alpha}\sigma'(s)\,\ud s=2(\sigma(0)-\sigma(\al))=\frac{2\rho(0)}{\al}.
\end{align*}
Let $c_{2}:=2\rho(0)$. Another application of Young's inequality yields
\begin{align*}
\norm{\psi'}_{p}&\leq \norm{\sigma'}_{1}\norm{\mathbf{1}_{[-(N+3\al+\beta),N+3\al+\beta]}}_{p}=\frac{c_{2}}{\alpha}(2(N+3\al+\beta))^{1/p},\\
\norm{\ph'}_{p}&\leq \frac{1}{2(\alpha+\beta)}\norm{\sigma'}_{1}\norm{\mathbf{1}_{[-(\al+\beta),\al+\beta]}}_{p}=\frac{c_{2}}{\alpha}(2(\al+\beta))^{-1/p}.
\end{align*}
Hence \eqref{norm estimate on interpolation space} becomes
\begin{align*}
\norm{U_{\mu}}_{\La(\D_{A}(\theta,q))}\leq M^{2}\left(1+\frac{c_{2}}{\al}\right)^{2}\left(\frac{N+3\al+\beta}{\al+\beta}\right)^{1/p}\norm{L_{\mu}}_{\La((\Ell^{p}(X),\W^{1,p}(X))_{\theta,q})}.
\end{align*}
Taking the infimum over $\alpha$ and $\beta$ yields
\begin{align}\label{besovless estimate bounded}
\norm{U_{\mu}}_{\La(\D_{A}(\theta,q))}\leq M^{2}\norm{L_{\mu}}_{\La((\Ell^{p}(X),\W^{1,p}(X))_{\theta,q})}.
\end{align}
Lemma \ref{besov spaces interpolation} now establishes \eqref{transference estimate bounded groups} and concludes the proof.
\end{proof}

\begin{remark}\label{constants}
Note that the constant $C$ in Proposition \ref{transference bounded groups} comes only from the equivalence of the norms on $(\Ell^{p}(X),\W^{1,p}(X))_{\theta,q}$ and $\B^{\theta}_{p,q}(X)$, whereas in Proposition \ref{transference unbounded groups} a constant is present which is inherent to the transference method. 
\end{remark}

\begin{remark}\label{sharpness transference}
Let $p\in[1,\infty)$ and let $(U(s))_{s\in\R}\subseteq\La(\Ell^{p}(\C))$ be the shift group given by $U(s)f(t):=f(t+s)$ for $f\in\Ell^{p}(\C)$, $s\in\R$ and almost all $t\in\R$. Then $(U(s)))_{s\in\R}$ is generated by $-\ui A$, where $Af:=\ui f'$ for $f\in\D(A)=\W^{1,p}(\C)$. Hence $\D_{A}(\theta,q)=(\Ell^{p}(\C),\W^{1,p}(\C))_{\theta,q}$ for $\theta\in(0,1)$ and $q\in[1,\infty]$. Moreover, for $\mu\in\eM(\R)$ and $f\in\Ell^{p}(\C)$, 
\begin{align*}
\int_{\R}U(s)f\,\ud\mu(s)=\mu\ast f=L_{\mu}(f).
\end{align*}
Hence, with $U_{\mu}$ as in \eqref{phillips calculus definition}, 
\begin{align*}
\norm{U_{\mu}}_{\La(\D_{A}(\theta,q))}=\norm{L_{\mu}}_{\La((\Ell^{p}(\C),\W^{1,p}(\C))_{\theta,q})}.
\end{align*}
This shows that \eqref{besovless estimate bounded} is sharp in general, up to possibly a change of constant. By Lemma \ref{besov spaces interpolation}, the same holds for \eqref{transference estimate bounded groups}.
\end{remark}

Corollary \ref{mikhlin condition} yields the following result, Corollary \ref{functional calculus real line introduction} from the Introduction.


\begin{corollary}\label{functional calculus real line}
Let $\theta\in(0,1)$ and $q\in[1,\infty]$. Then there exists a constant $C\geq 0$ such that the following holds. Let $-\ui A$ generate a $C_{0}$-group $(U(s))_{s\in\R}$ on a Banach space $X$ with $M:=\sup_{s\in\R}\norm{U(s)}<\infty$, and let $\mu\in\eM(\R)$ be such that $\F\mu\in\Ce^{1}\!(\R)$ with $\sup_{s\in\R}(1+\abs{s})\,\abs{(\F\mu)'(s)}<\infty$. Then
\begin{align*}
\norm{\int_{\R}U(s)x\,\mu(\ud s)}_{\theta,q}\leq CM^{2}\sup_{s\in\R}\,\abs{\F\mu(s)}+(1+\abs{s})\,\abs{(\F\mu)'(s)}\,\norm{x}_{\theta,q}
\end{align*}
for all $x\in\D_{A}(\theta,q)$.
\end{corollary}


\begin{remark}\label{other multiplier results}
To obtain Corollary \ref{functional calculus real line} we used Corollary \ref{mikhlin condition}, but there are other ways to verify the conditions of Proposition \ref{besov multiplier theorem}, for instance H\"{o}rmander type assumptions, cf.\cite[pp.~47-49]{Girardi-Weis}. More generally, one can define a norm on the space of all bounded Fourier multipliers $m$ on $\B^{r}_{p,q}(X)$ by $\norm{m}_{\mathcal{M}(\B^{r}_{p,q}(X))}:=\norm{T_{m}}_{\La(\B^{r}_{p,q}(X))}$, with $T_{m}$ as in \eqref{besov multiplier}. Proposition \ref{transference bounded groups} yields $\norm{U_{\mu}}_{\La(\D_{A}(\theta,q))}\leq C\norm{\F\mu}_{\mathcal{M}(\B^{r}_{p,q}(X))}$, which cannot be improved in general, cf.~Remark \ref{sharpness transference}.
\end{remark}

\begin{remark}\label{comparison with UMD case}
If $X$ is a UMD space then \eqref{classical transference principle} and the vector-valued Mikhlin multiplier theorem \cite[Theorem E.6.2 b]{Haase1} yield an estimate
\begin{align*}
\norm{\int_{\R}U(s)x\,\mu(\ud s)}_{X}\leq CM^{2}\norm{x}_{X}\sup_{s\in\R}\,\abs{\F\mu(s)}+\abs{s(\F\mu)'(s)}
\end{align*}
for all $x\in X$. Corollary \ref{functional calculus real line} then follows from \eqref{interpolation inequality}, and moreover singularities of $(\F\mu)'$ at zero are allowed. However, in our setting of general Banach spaces, the inhomogeneity of the Besov space $\B^{r}_{p,q}(X)$ implies that a condition at zero on the multiplier is needed to deal with the term $\ph_{0}m$ in Proposition \ref{besov multiplier theorem}.
\end{remark}

\begin{remark}\label{functional calculus real line remark}
Letting $f:=\F\mu$, Corollary \ref{functional calculus real line} yields an estimate
\begin{align}\label{functional calculus real line estimate}
\norm{f(A_{\theta,q})}\leq C\sup_{s\in\R}\,\abs{f(s)}+(1+\abs{s})\,\abs{f'(s)}.
\end{align}
This is a functional calculus statement for $A_{\theta,q}$ involving functions on the real line. One may now ask to which functions $f$ on the real line the definition of $f(A_{\theta,q})$ can be extended in a sensible manner such that \eqref{functional calculus real line estimate} holds. We can take the closure of the Fourier transforms of measures in the space consisting of all functions $f\in\Ce^{1}\!(\R)$ for which $\sup_{s\in\R}\,\abs{f(s)}+(1+\abs{s})\,\abs{f'(s)}$ is finite, or approximate by holomorphic functions as in \cite[Lemma 4.15]{Kriegler}, using Theorem \ref{main functional calculus result}. This will yield a definition of $f(A_{\theta,q})$ for a class of functions on the real line and a bound as in \eqref{functional calculus real line estimate}, but the question then remains how this definition relates to other known extensions of functional calculi. In the present article we restrict ourselves to results about holomorphic functional calculi.
\end{remark}

\section{Functional calculus results}\label{functional calculus results}

We now use the theory established in the previous sections to prove our main functional calculus result, Theorem \ref{main functional calculus result introduction}. Recall the definition of the analytic Mikhlin algebra $\HT^{\infty}_{1}\!(\St_{\w})$ from \eqref{analytic mikhlin algebra}.


\begin{theorem}\label{main functional calculus result}
Let $-\ui A$ be the generator of a $C_{0}$-group $\left(U(s)\right)_{s\in\R}$ on a Banach space $X$ and let $\theta\in(0,1)$, $q\in[1,\infty]$ and $\w>\theta(U)$ be given. Then there exists a constant $C\geq 0$ such that $f(A_{\theta,q})\in\La(\D_{A}(\theta,q))$ with
\begin{align*}
\norm{f(A_{\theta,q})}_{\La(\D_{A}(\theta,q))}\leq C\norm{f}_{\HT^{\infty}_{1}\!(\St_{\w})}
\end{align*}
for all $f\in\HT^{\infty}_{1}\!(\St_{\w})$. If $(U(s))_{s\in\R}$ is uniformly bounded then $C$ can be chosen independent of $\w>0$.
\end{theorem}
\begin{proof}
First consider $f\in \HT^{\infty}_{1}\!(\St_{\w})\cap \E(\St_{\w})$ and fix $\alpha\in(\theta(U),\w)$ and $p\in[1,\infty)$. By Lemma \ref{decay implies fourier transform} there exists $\mu\in\eM_{\alpha}\!(\R)$ such that $f=\F\mu$. By Lemmas \ref{decay implies fourier transform} and \ref{functional calculus on interpolation spaces} and Proposition \ref{transference unbounded groups},
\begin{align}\label{convolution estimate}
\norm{f(A_{\theta,q})}=\norm{(U_{\mu})_{\theta,q}}\leq C_{1}\norm{L_{\mu_{\alpha}}}_{\La(\B^{\theta}_{p,q}(X))}=C_{1}\norm{T_{\F\mu_{\alpha}}}_{\La(\B^{\theta}_{p,q}(X))}
\end{align}
for some constant $C_{1}\geq 0$, where $T_{\F\mu_{\alpha}}$ is as in \eqref{besov multiplier}. Since
\begin{align*}
\F\mu_{\alpha}(s)=\frac{f(s+i\alpha)+f(s-i\alpha)}{2}\qquad(s\in\R),
\end{align*}
Corollary \ref{mikhlin condition} yields a constant $C_{2}\geq 0$ such that
\begin{align}\label{estimate for elementary functions}
\norm{f(A_{\theta,q})}\leq C_{2}\sup_{s\in\R}\,\abs{\F\mu_{\alpha}(s)}+(1+\abs{s})\abs{(\F\mu_{\alpha})'(s)}\leq C_{2}\norm{f}_{\HT^{\infty}_{1}\!(\St_{\w})}.
\end{align}
For general $f\in\HT^{\infty}_{1}\!(\St_{\w})$ first assume that $q<\infty$. By part b) of Lemma \ref{functional calculus on interpolation spaces}, $\D(A_{\theta,q})$ is dense in $\D_{A}(\theta,q)$. Let $\tau_{k}(z):=-k^{2}(ik-z)^{-2}$ for $k\in\N$ with $k>\w$ and $z\in \St_{\w}$. Then $\tau_{k}, f\tau_{k}\in \HT^{\infty}_{1}\!(\St_{\w})\cap \E(\St_{\w})$,
\begin{align*}
\sup_{k} \norm{f\tau_{k}}_{\HT^{\infty}_{1}\!(\St_{\w})}\leq \norm{f}_{\HT^{\infty}_{1}\!(\St_{\w})}\sup_{k} \norm{\tau_{k}}_{\HT^{\infty}_{1}\!(\St_{\w})}<\infty
\end{align*}
and $f\tau_{k}(z)\rightarrow f(z)$ as $k\rightarrow \infty$, for all $z\in\St_{\w}$. Now \eqref{estimate for elementary functions} yields
\begin{align*}
\norm{f\tau_{k}(A_{\theta,q})}\leq C_{2}\norm{f\tau_{k}}_{\HT^{\infty}_{1}\!(\St_{\w})}\leq C\norm{f}_{\HT^{\infty}_{1}\!(\St_{\w})}
\end{align*}
for some $C\geq 0$. Hence the Convergence Lemma \ref{convergence lemma} implies $f(A)\in \La(X)$ and
\begin{align}\label{main functional calculus estimate}
\norm{f(A_{\theta,q})}\leq C\norm{f}_{\HT^{\infty}_{1}\!(\St_{\w})}.
\end{align}
Finally, for $q=\infty$ the Reiteration Theorem \cite[Theorem 3.5.3]{Bergh-Lofstrom} yields
\begin{align*}
\D_{A}(\theta,\infty)=\left(\D_{A}(\theta_{1},1),\D_{A}(\theta_{2},1)\right)_{\theta_{3},\infty}
\end{align*}
with equivalence of norms, where $\theta_{1},\theta_{2},\theta_{3}\in (0,1)$ are such that $\theta_{1}\neq \theta_{2}$ and $\theta_{1}(1-\theta_{3})+\theta_{2}\theta_{3}=\theta$. Combining \eqref{main functional calculus estimate} and \eqref{interpolation inequality} concludes the proof of the first statement.

In the case where $(U(s))_{s\in\R}$ is uniformly bounded, use  Proposition \ref{transference bounded groups} instead of \ref{transference unbounded groups} in \eqref{convolution estimate} to obtain 
\begin{align*}
\norm{f(A_{\theta,q})}\leq C_{1}\norm{T_{\F\mu}}_{\La(\B^{\theta}_{p,q}(X))}
\end{align*}
for all $f\in \HT^{\infty}_{1}\!(\St_{\w})\cap \E(\St_{\w})$ and some constant $C_{1}\geq 0$ independent of $\w$. The rest of the proof is the same as before.
\end{proof}


\begin{remark}\label{original result unbounded groups}
Compare Theorem \ref{main functional calculus result} with Theorem 3.6 in \cite{Haase5}. There an estimate
\begin{align}\label{original estimate}
\norm{f(A)}_{\La(X)}\leq C\!\sup_{z\in\St_{\w}}\abs{f(z)}+\abs{zf'(z)}
\end{align}
is obtained when the underlying space $X$ is a UMD space, and the constant $C$ is independent of $\w$ when the group in question is uniformly bounded. Theorem \ref{main functional calculus result} follows from \eqref{original estimate} by interpolation, and this seems to yield a stronger result since the term $\sup_{z\in\St_{\w}} \abs{f'(z)}$ does not appear in \eqref{original estimate}. In fact, the norms $\sup_{z\in\St_{\w}}\abs{f(z)}+\abs{zf'(z)}$ and $\norm{f}_{\HT^{\infty}_{1}\!(\St_{\w})}$ are equivalent, since $0\in\St_{\w}$ for all $\w>0$. So for generators of unbounded groups \eqref{original estimate} does not yield an essentially better estimate than Theorem \ref{main functional calculus result}. This is different for generators of uniformly bounded groups, since the norm equivalence of $\sup_{z\in\St_{\w}}\abs{f(z)}+\abs{zf'(z)}$ and $\norm{f}_{\HT^{\infty}_{1}\!(\St_{\w})}$ fails as $\w\downarrow 0$. Hence for generators of uniformly bounded groups \eqref{original estimate} yields a strictly stronger result on $\D_{A}(\theta,q)$ than Theorem \ref{main functional calculus result}.
\end{remark}

\begin{remark}\label{fractional domain inclusion}
Let $\lambda\in\C$ with $\Real(\lambda)>\w$. By \cite[Corollary 6.6.3]{Haase1}, $\D((\lambda-\ui A)^{\alpha})\subseteq \D_{A}(\alpha,\infty)$ for each $\alpha\in(0,1)$. Hence Theorem \ref{main functional calculus result} yields $f(A)(\lambda-\ui A)^{-\alpha}\in\La(X)$ for all $\w>\theta(U)$, $f\in\HT^{\infty}_{1}(\St_{\w})$ and $\alpha>0$. However, this already follows from \cite[Proposition 8.2.3]{ABHN2011} in a similar manner as in \cite[Remark 5.2]{Haase-Rozendaal2013}. Moreover, using arguments as in \cite[Remark 3.9]{Haase-Rozendaal2013}, \cite[Proposition 8.2.3]{ABHN2011} already implies that $f(A):\D_{A}(\theta,q)\rightarrow \D_{A}(\theta',q')$  is bounded for all $\theta'<\theta$ and $q,q'\in[1,\infty]$. The improvement that Theorem \ref{main functional calculus result} provides lies in going from $\theta'<\theta$ to $\theta'=\theta$.
\end{remark}

\begin{remark}\label{other multiplier results 2}
As already noted in Remark \ref{other multiplier results}, we could have used Fourier multiplier results on Besov spaces other than Corollary \ref{mikhlin condition}. These lead to statements about the boundedness of functional calculi for other function algebras.
\end{remark}

For $\ph\in(0,\pi)$ define
\begin{align}\label{sector}
\Se_{\ph}:=\left\{z\in\C\mid \abs{\arg(z)}<\ph\right\},
\end{align}
and for $\psi\in(0,\pi/2)$ and $\w>0$,
\begin{align*}
\Sigma_{\psi}:=\Se_{\psi}\cup -\Se_{\psi},\qquad \mathrm{V}_{\psi,\w}:=\St_{\w}\cup \Sigma_{\psi}.
\end{align*}

\begin{lemma}\label{no derivative}
Let $\w>\w'>0$ and $\psi\in(0,\pi/2)$. Then $\HT^{\infty}\!(\mathrm{V}_{\w,\psi})$ is continuously embedded in $\HT^{\infty}_{1}\!(\St_{\w'})$.
\end{lemma}
\begin{proof}
This follows in a straightforward manner from Lemma 4.5 in \cite{Haase5}.
\end{proof}

\begin{corollary}\label{venturi region result}
Let $-\ui A$ be the generator of a $C_{0}$-group $(U(s))_{s\in\R}$ on a Banach space $X$ and let $\theta\in(0,1)$ and $q\in[1,\infty]$. Then $A_{\theta,q}$ has a bounded $\HT^{\infty}\!(\mathrm{V}_{\w,\psi})$-calculus for all $\w>\theta(U)$ and $\psi\in(0,\pi/2)$.
\end{corollary}

So far we have considered functional calculus on interpolation spaces for the couple $(X,\D(A))$. The next corollary extends our results to other interpolation couples.

\begin{corollary}\label{other interpolation spaces}
Let $-\ui A$ be the generator of a $C_{0}$-group $\left(U(s)\right)_{s\in\R}$ on a Banach space $X$ and let $\theta\in(0,1)$, $q\in[1,\infty]$ and $m,n\in\N_{0}$ with $m\neq n$. Then the part of $A$ in $(\D(A^{m}),\D(A^{n}))_{\theta,q}$ has a bounded $\HT^{\infty}_{1}\!(\St_{\w})$-calculus for all $\w>\theta(U)$. If $(U(s))_{s\in\R}$ is uniformly bounded then the constant bounding the calculus is independent of $\w>0$.
\end{corollary}
\begin{proof}
First note that since
\begin{align*}
(\D(A^{m}),\D(A^{n}))_{\theta,q}=(\D(A^{n}),\D(A^{m}))_{1-\theta,q}
\end{align*}
by \cite[Theorem 3.4.1]{Bergh-Lofstrom}, we may assume that $m<n$. Using the similarity transform $R(\lambda,A)^{m}:X\rightarrow \D(A^{m})$, it suffices to let $m=0$. Suppose that $n\theta\notin \N$. By Lemma 3.1.3 and Proposition 3.1.8 in \cite{Lunardi},
\begin{align*}
(X,\D(A^{n}))_{\theta,q}=(\D(A^{k}),\D(A^{k+1}))_{\theta',q}
\end{align*}
for some $k\in\N_{0}$ and $\theta'\in(0,1)$. Another similarity transform shows that we can let $k=0$. Now Theorem \ref{main functional calculus result} yields the statement. 

If $k:=n\theta\in\N$, the Reiteration Theorem \cite[Theorem 3.5.3]{Bergh-Lofstrom} yields
\begin{align*}
(X,\D(A^{n}))_{\theta,q}=\left((\D(A^{k-1}),\D(A^{k}))_{1/2,q}, (\D(A^{k}),\D(A^{k+1}))_{1/2,q}\right)_{1/2,q}.
\end{align*}
By what we have already shown and \eqref{interpolation inequality}, this concludes the proof.
\end{proof}

\section{Additional results}\label{additional results}

We now deduce several applications of Theorem \ref{main functional calculus result}. Corollary \ref{other interpolation spaces} can be applied in this section to yield results for other interpolation couples.

We first state a proposition about the convergence of certain principal value integrals, an interpolation version of \cite[Theorem 4.4]{Haase5} on general Banach spaces. If $g\in \Ell^{1}[-1,1]$ is even then by $\mathrm{PV}-g(s)/s$ we mean the distribution defined by
\begin{align*}
\langle \mathrm{PV}-g(s)/s,\ph\rangle:=\lim_{\epsilon\searrow 0}\int_{\epsilon\leq|s|\leq1}g(s)\ph(s)\frac{\ud s}{s}
\end{align*}
for $\ph\in \Ce^{\infty}\!(\R)$ compactly supported. By $\mathrm{BV}[-1,1]$ we denote the functions of bounded variation on $[-1,1]$. 

\begin{proposition}\label{principal value result}
Let $-\ui A$ be the generator of a $C_{0}$-group $(U(s))_{s\in\R}$ on a Banach space $X$. Let $g\in \mathrm{BV}[-1,1]$ be even and set $f:=\mathcal{F}(\mathrm{PV}-g(s)/s)$. Then $f(A_{\theta,q})\in\La(\D_{A}(\theta,q))$ and
\begin{align}\label{principal value identity}
f(A)x=\lim_{\epsilon\searrow 0}\int_{\epsilon\leq |s|\leq 1}g(s)U(s)x\,\frac{\ud s}{s}
\end{align}
for all $\theta\in(0,1)$, $q\in[1,\infty)$ and $x\in \D_{A}(\theta,q)$.
\end{proposition}
\begin{proof}
By \cite[Lemma 4.3]{Haase5}, $f\in\HT^{\infty}_{1}\!(\St_{\w})$ for all $\w>0$. Theorem \ref{main functional calculus result} now yields the first statement. For \eqref{principal value identity} we may let $q<\infty$, since $\D_{A}(\theta,\infty)\subseteq\D_{A}(\theta',1)$ for $\theta'<\theta$ \cite[Proposition 1.1.4]{Lunardi}. Now use the Convergence Lemma as in the proof of \cite[Theorem 4.4]{Haase5}.
\end{proof}

\begin{remark}\label{matter of convergence}
Convergence in \eqref{principal value identity} takes place in $\D_{A}(\theta,q)$ for $q<\infty$. For $q=\infty$ the limit and the integral converge in $X$ and in $\D_{A}(\theta',q')$ for $\theta'<\theta$ and $q'\in[1,\infty)$. Compare with Remark \ref{types of convergence}.
\end{remark}


\subsection{Results for sectorial operators and cosine functions}\label{sectorial and cosine}

An operator $A$ on a Banach space $X$ is \emph{sectorial} of angle $\ph\in(0,\pi)$ if $\sigma(A)\subseteq \overline{\Se_{\ph}}$, where $\Se_{\ph}$ is as in \eqref{sector}, and if $\sup\left\{\norm{zR(z,A)}\,\mid\psi\in\C\setminus \Se_{\psi}\right\}<\infty$ for all $\psi\in (\ph,\pi)$. A functional calculus for sectorial operators can be constructed by a method similar to the one used for strip-type operators. For details see \cite[Chapter 2]{Haase1}. 

If $A$ is an injective sectorial operator of angle $\ph\in(0,\pi)$ then $\log(A)$ is defined, as is $f(A)$ for all $f\in \HT^{\infty}\!(\Se_{\psi})$ and $\psi\in(\ph,\pi)$. A sectorial operator $A$ has \emph{bounded imaginary powers} if $A$ is injective and if $-\ui\log(A)$ is the generator of a $C_{0}$-group $(U(s))_{s\in\R}$ on $X$. Then $U(s)=A^{-\ui s}$ for all $s\in\R$, and we write $A\in \mathrm{BIP}(X)$. Moreover, $A$ is sectorial of angle $\theta_{A}:=\theta(U)$, by \cite[Corollary 4.3.4]{Haase1}.

For $\psi\in(0,\pi)$ define $\HT^{\infty}_{\log}(\Se_{\psi})$ to be the unital Banach algebra of all $f\in\HT^{\infty}\!(\Se_{\psi})$ for which 
\begin{align*}
\norm{f}_{\HT^{\infty}_{\log}(\Se_{\psi})}:=\sup_{z\in \Se_{\psi}}\abs{f(z)}+(1+\abs{\log(z)})\abs{z f'(z)}<\infty,
\end{align*}
endowed with the norm $\norm{\cdot}_{\HT^{\infty}_{\log}(\Se_{\psi})}$. 

\begin{proposition}\label{sectorial result}
Let $X$ be a Banach space and $A\in \mathrm{BIP}(X)$ such that $\theta_{A}<\pi$. Let $\theta\in(0,1)$ and $q\in[1,\infty]$. Set $Y:=(X,\D(\log(A)))_{\theta,q}$. Then $A_{Y}$ has a bounded $\HT^{\infty}_{\log}(\Se_{\psi})$-calculus on $Y$ for all $\psi\in(\theta_{A},\pi)$. If $\sup_{s\in\R}\norm{A^{\ui s}}<\infty$ then the constant bounding the calculus is independent of $\psi>0$.
\end{proposition}
\begin{proof}
Let $\psi\in(\theta_{A},\pi)$ be given and note that $f\mapsto f\circ\log$ is an isometric algebra isomorphism $\HT^{\infty}_{1}\!(\St_{\psi})\rightarrow \HT^{\infty}_{\log}(\Se_{\psi})$. By Lemma \ref{functional calculus on interpolation spaces} as well as Theorem 4.2.4 and Proposition 6.1.2 from \cite{Haase1},
\begin{align*}
f(\log(A)_{Y})=f(\log(A))_{Y}=(f\circ\log)(A)_{Y}=(f\circ \log)(A_{Y})
\end{align*}
for all $f\in \HT^{\infty}_{1}\!(\St_{\psi})$. Now Theorem \ref{main functional calculus result} concludes the proof.
\end{proof}

\begin{remark}\label{Dore's result}
Let $A$ be an injective sectoral operator of angle $\ph\in(0,\pi)$, and let $\alpha>0$, $\theta\in(0,1)$ and $q\in[1,\infty]$. By \cite[Corollary 6.6.3]{Haase1}, a special case of which was proved by Dore \cite[Theorem 3.2]{Dore2001}, the part of $A$ in $(X,\D(A^{\alpha})\cap \Ri(A^{\alpha}))_{\theta,q}$ has a bounded $\HT^{\infty}(\Se_{\psi})$-calculus for all $\psi\in(\ph,\pi)$. Here $\Ri(A)$ is the range of $A$. By \cite[Corollary 6.6.3]{Haase1} and because $\log(A)A^{\alpha\theta}(1+A)^{-2\alpha\theta}\in\La(X)$,
\begin{align*}
(X,\D(A^{\alpha})\cap \Ri(A^{\alpha}))_{\theta,q}\subseteq (X,\D(A^{\alpha}))_{\theta,q}\subseteq \D(A^{\alpha\theta})\subseteq\D(\log(A)),
\end{align*}
and in general $\D(\log(A))$ is strictly included in $(X,\D(\log(A)))_{\theta',q'}$ for all $\theta'\in(0,1)$ and $q'\in[1,\infty]$. Hence the result of Dore does not imply Proposition \ref{sectorial result}.
\end{remark}

A \emph{cosine function} $\mathrm{Cos}:\R\rightarrow \La(X)$ on a Banach space $X$ is a strongly continuous mapping such that $\mathrm{Cos}(0)=I$ and
\begin{align*}
\mathrm{Cos}(t+s)+\mathrm{Cos}(t-s)=2\mathrm{Cos}(t)\mathrm{Cos}(s)
\end{align*}
for all $s,t\in\R$. Then
\begin{align*}
\theta(\mathrm{Cos}):=\inf\left\{\w\geq 0\left|\exists M\geq 0:\norm{\mathrm{Cos}(t)}\leq M\ue^{\w|t|}\textrm{ for all }t\in\R\right.\right\}<\infty.
\end{align*}
The \emph{generator} of a cosine function is the unique operator $-A$ on $X$ that satisfies
\begin{align*}
\lambda R(\lambda^{2},-A)=\int_{0}^{\infty}\ue^{-\lambda t}\mathrm{Cos}(t)\,\ud t
\end{align*}
for $\lambda>\theta(\mathrm{Cos})$. Then $A$ is an operator of \emph{parabola-type} $\w=\theta(\mathrm{Cos})$. This means that $\sigma(A)\subseteq \overline{\varPi_{\w}}$, where $\varPi_{\w}:=\left\{z^{2}\mid z\in \St_{\w}\right\}$, and that for all $\w'>\w$ there exists $M_{\w'}\geq 0$ such that
\begin{align*}
\norm{R(\lambda,A)}\leq \frac{M_{\w'}}{\sqrt{\abs{\lambda}}\left(\abs{\Imag(\sqrt{\lambda})}-\w'\right)}
\end{align*}
for all $\lambda\notin \varPi_{\w'}$. For such operators there is a natural functional calculus, as before, and a version of Lemma \ref{functional calculus on interpolation spaces} holds. For details see \cite{Haase13}. For $\w>0$ let
\begin{align*}
\HT^{\infty}_{1}\!(\varPi_{\w}):=\left\{f\in \HT^{\infty}\!(\varPi_{\w})\left| \norm{f}_{\HT^{\infty}_{1}\!(\varPi_{\w})}:=\sup_{z\in\varPi_{\w}}\abs{f(z)}+(1+\abs{z})\,\abs{f'(z)}<\infty\right.\right\},
\end{align*}
a Banach algebra when endowed with the norm $\norm{\cdot}_{\HT^{\infty}_{1}\!(\varPi_{\w})}$.

\begin{proposition}\label{cosine function result}
Let $-A$ be the generator of a cosine function $\mathrm{Cos}$ on a Banach space $X$ and let $\theta\in(0,1)$, $q\in[1,\infty]$. Then the part $A_{\theta,q}$ of $A$ in $\D_{A}(\theta,q)$ has a bounded $\HT^{\infty}_{1}\!(\varPi_{\w})$-calculus for all $\w>\theta(\mathrm{Cos})$. If $\sup_{s\in\R}\norm{\mathrm{Cos}(s)}<\infty$ then the constant bounding the calculus is independent of $\w>0$.
\end{proposition}
\begin{proof}
We mainly follow \cite[Theorem 5.3]{Haase5}, providing extra details where necessary. There is a unique subspace $V\subseteq X$, the \emph{Kisy\'{n}ski space}, such that the operator $-\ui\mathcal{A}$,
\begin{align*}
\mathcal{A}:=\ui\left[\begin{array}{cc}0&\I_{V}\\-A&0\end{array}\right],
\end{align*}
with domain $\D(\mathcal{A}):=\D(A)\times V$, generates a $C_{0}$-group $(U(s))_{s\in\R}$ on $X\times V$. Moreover, $\theta(\mathrm{Cos})=\theta(U)$ \cite[Theorem 6.2]{Haase4}. 

Let $\w>\theta(\mathrm{Cos})$. Then $f\in \HT^{\infty}\!(\varPi_{\w})$ is an element of $\HT^{\infty}_{1}\!(\varPi_{\w})$ if and only if $[z\mapsto g(z):=f(z^{2})]\in \HT^{\infty}_{1}\!(\St_{\w})$, with $\norm{g}_{\HT^{\infty}_{1}\!(\St_{\w})}\leq 4\norm{f}_{\HT^{\infty}_{1}\!(\varPi_{\w})}$. Moreover, $f(A)\oplus f(A_{V})=g(\mathcal{A})$ and
\begin{align*}
f(A_{\theta,q})\oplus f(A_{V})=(f(A)\oplus f(A_{V}))_{Y}=g(\mathcal{A})_{Y}=g(\mathcal{A}_{Y})
\end{align*}
for all $f\in \HT^{\infty}_{1}\!(\varPi_{\w})$, where $Y:=\D_{\mathcal{A}}(\theta,q)=\D_{A}(\theta,q)\times V$. Hence Theorem \ref{main functional calculus result} concludes the proof.
\end{proof}

\subsubsection{Acknowledgements}

We would like to thank our colleagues for interesting discussions and helpful ideas.

\bibliographystyle{plain}
\bibliography{Bibliografie}

\end{document}